\documentclass{amsart}
\usepackage[mathscr]{eucal}
\usepackage{amsmath,amssymb,hyperref}
\usepackage{amstext}
\usepackage{graphics}
\usepackage{xcolor}
\usepackage{mathrsfs}
\usepackage{epstopdf}

\usepackage{tikz}
\usetikzlibrary{matrix}

\usepackage[all]{xy}

\usepackage[T1]{fontenc}
\usepackage[utf8]{inputenc}

\newtheorem{thm}{Theorem}[section]

\newtheorem{cor}[thm]{Corollary}
\newtheorem{prop}[thm]{Proposition}

\newtheorem{lemma}[thm]{Lemma}

\theoremstyle{definition}

\newtheorem{definition}[thm]{Definition}
\newtheorem{remark}[thm]{Remark}
\newtheorem{example}[thm]{Example}

\def\X0{X^{\circ}}

\def\Y0{Y^{\circ}}

\input xy
\xyoption{all}

\addtocounter{section}{0}             
\numberwithin{equation}{section}       

\begin{document}

\title[Darboux-Jouanolou Integrability over Arbitrary Fields]
{Darboux-Jouanolou Integrability over Arbitrary Fields}


\author[E. A. Santos]{Edileno de Almeida SANTOS}
\address{Faculdade de Ciências Exatas e Tecnologia (FACET) - Universidade Federal da Grande Dourados (UFGD), Rodovia Dourados - Itahum, Km 12 - Cidade Universitátia, Dourados - MS,
Brazil}
\email{edilenosantos@ufgd.edu.br}

\author[S. Rodrigues]{Sergio RODRIGUES}
\address{Faculdade de Ciências Exatas e Tecnologia (FACET) - Universidade Federal da Grande Dourados (UFGD), Rodovia Dourados - Itahum, Km 12 - Cidade Universitátia, Dourados - MS,
Brazil}
\email{sergiorodrigues@ufgd.edu.br}



\begin{abstract}
We prove a Darboux-Jouanolou type theorem on the algebraic integrability of polynomial differential $1$-forms over arbitrary fields.
\end{abstract}

\maketitle

\setcounter{tocdepth}{1}
\sloppy



\section{Introduction}
Ordinary differential equations appear in many branches of Mathematics and its applications. For example, occasionally the differential equations are given by a differential $1$-form or a vector field on $\mathbb R^n$ or $\mathbb C^n$, and in that case to better understand their solutions (in real or complex time, respectively) we can search for a {\it first integral} (a constant function along the solutions).

In that sense, the Darboux theory of integrability is a classical approach. The foundations of this was given by J.-G. Darboux (1878) in his memory \cite{Darboux}, showing how can be constructed a first integral for a polynomial vector field, on $\mathbb R^2$ or $\mathbb C^2$, from a large enough number of invariant algebraic curves. In fact, he proved that, if $d$ is the degree of the planar vector field, with $\binom{d+1}{2}+1$ invariant algebraic curves we can compute a first integral using that curves. H. Poincaré (1891) noticed the difficulty to obtain algorithmically that invariant curves (see \cite{Poincaré}).

Improving the Darboux's result, J.-P. Jouanolou (1979) showed in \cite{Jouanolou} that if the number of invariant curves for a planar vector field of degree $d$ on  $\mathbb R^2$ or $\mathbb C^2$ is at least $\binom{d+1}{2}+2$, then it admits a rational first integral, and it can be computed from the invariant curves. In fact, Jouanolou's Theorem says that if $K$ is a field of characteristic $0$ and $\omega$ is a polynomial $1$-form of degree $d$ on $K^n$ with at least $\binom{d-1+n}{n}\cdot \binom{n}{2}+2$ invariant irreducible algebraic hypersurfaces, then $\omega$ has a rational first integral, and that one can also be computed in terms of the invariant hypersurfaces.

The {\it Darboux Integration Method} can be used in Physics. For example, C. G. Hewitt (1991) in \cite{Hewitt} reveals by this method some new solutions to the Einstein field equations. Many others physical models have been successfully studied in this framework (see \cite{Llibre1}, \cite{Llibre2}, \cite{Valls}, \cite{Zhang}).

E. Ghys (2000) in \cite{Ghys} extended Jouanolou's result to holomorphic foliations on compact complex manifolds. In the local analytic context, under some assumptions, B. Scárdua (2011) shows in \cite{Scárdua} a Darboux-Jouanolou type theorem for germs of integrable $1$-forms on $(\mathbb C^n,0)$. Over fields of characteristic zero, a Darbou-Jouanolou type theorem is proved in \cite{Corrêa} for polynomial $r$-forms, $r\geq 0$.

M. Brunella and M. Nicollau (1999) proved in \cite{B} a positive characteristic version of Darboux-Jouanolou theorem: if $\omega$ is a rational $1$-form on a smooth projective variety over a field $K$ of positive characteristic $p>0$ with infinitely many invariant hypersurfaces, then $\omega$ admits a rational first integral. The proof uses different tools of algebraic geometry and it is not constructive. 

J. V. Pereira (2001) shows in \cite{JVP} that a generic vector field on an affine space of positive characteristic admits an invariant algebraic hypersurface. This theorem is in sharp contrast with characteristic $0$ case where a theorem of Jouanolou says that a generic vector field on the complex plane does not admit any invariant algebraic curve.

Our goal in this paper is to give a general account of Darboux-Jouanolou integrability for polynomial $1$-forms on $K^n$, where $K$ is an arbitrary field. In particular, we obtain a general version of {\it Darboux-Jouanolou Criterion} (\cite{Jouanolou},  Théorème 3.3, p. 102). For each polynomial $r$-form $\omega$ of degree $d$ on $K^n$, with $K$ of characteristic $p\geq 0$, we define a natural number $N_K(n,d,r)$ (see Definition \ref{D:Nk}) that depends only on $n$, $r$, $d$ and $p$, and we prove the following

\begin{thm}\label{Main}
Let $\omega\in \Omega^1(K^n)$ be a polynomial $1$-form of degree $d$ over an arbitrary field $K$. If $\omega$ has $N_K(n,d-1,2)+2$ invariant hypersurfaces, then $\omega$ admits a rational first integral.
\end{thm}

The number $N_K(n,d-1,2)+2$ has the property that in characteristic $0$ holds $N_K(n,d-1,2)+2= \binom{d-1+n}{n}\cdot \binom{n}{2}+2$ but in characteristic $p>0$ we have
$$
N_K(n,d-1,2)+2< \binom{d-1+n}{n}\cdot \binom{n}{2}+2
$$

\section{The Grassmann Algebra}

For the generalities about Grassmann Algebra see  Chapter 2 in \cite{Warner}.

Let $V$ be a vector space over an arbitrary field $K$.

\begin{definition}
A $r$-linear function ({\it $r$-covector}) $f:V^r\rightarrow K$ is said to be {\it symmetric} if
$$
f(v_{\sigma(1)},..., v_{\sigma(r)})=f(v_1,...,v_r)
$$
and it is said to be {\it alternanting} if
$$
f(v_{\sigma(1)},..., v_{\sigma(r)})=(sgn \sigma)f(v_1,...,v_r)
$$
for all permutations $\sigma \in S_r$.
\end{definition}

The $K$-vector space $A_r(V)$ is the space of all alternating $r$-linear functions on $V$.

\begin{definition}
Given a $r$-linear function $f:V^r\rightarrow K$, we define
$$
(Sf)(v_1,...,v_r)=\sum_{\sigma \in S_r}f(v_{\sigma(1)},..., v_{\sigma(r)})
$$
as the {\it symmetrizing operator}; and 
$$
(Af)(v_1,...,v_r)=\sum_{\sigma \in S_r}(sgn \sigma)f(v_{\sigma(1)},..., v_{\sigma(r)})
$$
as the {\it alternating operator}.
\end{definition}

\begin{definition}
Let $f$ be a $r$-linear function and $g$ a $s$-linear function on the $K$-vector space $V$. Their {\it tensor product} is the ($r+s$)-linear function $f\otimes g$ defined by
$$
(f\otimes g)(v_1,...,v_{r+s})=f(v_{1},..., v_{r})g(v_{r+1},..., v_{r+s})
$$
\end{definition}

\begin{definition}
Let $f\in A_r(V)$ and $g\in A_s(V)$. We define their {\it wedge product}, also called {\it exterior product}, by
$$
f\wedge g =A(f\otimes g)
$$
\end{definition}

\begin{remark}
This wedge product just defined can be different by a constant factor from the usual one in differential geometry, and we use that in order to make sense in a field of positive characteristic.
\end{remark}

For a finite-dimensional vector space $V$, say of dimension $n$, define
$$
A_{*}(V)=\bigoplus_{r=0}^\infty A_r(V)=\bigoplus_{r=0}^n A_r(V)
$$
With the wedge product of {\it multicovectors} as multiplication, $A_*(V)$ becomes an anticommutative graded algebra, called the {\it exterior algebra} or the {\it Grassmann algebra} of multicovectors on the vector space $V$.

\section{Rational and Polynomial Vector Fields}
Let $K$ be an arbitrary field. Denote by $K^n$ the $n$-dimensional $K$-vector space (with coordinates $z_1$, $z_2$,..., $z_n$). Also denote by $K[z]$ the polynomial ring $K[z_1,...,z_n]$ and by $K(z)$ the field of fractions $K(z_1,...,z_n)$. 


\begin{definition}
A {\it rational vector field} $X$ on $K^n$ is a $K$-derivation of $K(z)$, that is, a $K$-linear operator on $K(z)$ that satisfies Leibniz's rule, i. e.,
$$
X(fg)=fX(g)+gX(f)
$$
\end{definition}

A vector field $X$ can be written as
$$
X=\sum_{i=1}^n X(z_i)\frac{\partial}{\partial z_i}
$$
where $X(z_i)=X_i\in K(z)$ and $\frac{\partial}{\partial z_i}$ is de natural derivation such that $\frac{\partial}{\partial z_i}(z_j)=\delta_{ij}$ (the {\it Kronecker's delta} is $\delta_{ij}=1$ if $i=j$ and $\delta_{ij}=0$ if $i\neq j$).

If $X(z_i)=X_i\in K[z]$ is polynomial, $i=1$,..., $n$, then we say that $X$ is a {\it polynomial vector field} on $K^n$, and in this case it is a $K$-derivation of $K[z]$. The $K$-vector space of polynomial vector fields will be denoted by $\mathfrak{X}(K^n)$.

In case of $K=\mathbb R$ or $\mathbb C$, the system
$$
\frac{dz_i}{dt}=X_i(z)
$$
$i=1,...,n$, defines a system of differential equations on $K^n$.


\begin{definition}
Like in theory of differential equations if $f\in K(z)$ is a rational function such that $df\neq 0$ and $X(f)=0$ we say that $f$ is a {\it first integral} or a {\it non-trivial constant of derivation} of $X$. In case of $K=\mathbb R$ or $\mathbb C$ the leaves $\{ f(z)=constant \}$ are invariant sets for the system $\frac{dz_i}{dt}=X_i(z)$, $i=1,...,n$. We say that an irreducible polynomial  $F \in K[z]$ is {\it invariant} by a polynomial vector field $X$ if $F$ divides $X(F)$ (we also say that $\{F=0\}\subset (K^a)^n$ is an {\it invariant hypersurface} for $X$ over the algebraic closure $K^a$ of $K$). In case of $K=\mathbb R$ or $\mathbb C$, $\{ F=0 \}$ is an invariant set for the system $\frac{dz_i}{dt}=X_i(z)$, $i=1,...,n$.
\end{definition}



\section{Rational and Polynomial Differential Forms}


A {\it rational differential $1$-form} $\omega$ on $K^n$ is a $K(z)$-linear map from the set of rational vector fields to $K(z)$:
$$
\omega=\sum_{i=1}^n \omega(\frac{\partial}{\partial z_i})d z_i
$$
where $\omega(\frac{\partial}{\partial z_i})\in K(z)$ and  $dz_i$ is de natural $1$-form dual to $\frac{\partial}{\partial z_i}$, that is, $dz_i(\frac{\partial}{\partial z_j})=\delta_{ij}$.

Hence $dz_1$,..., $dz_n$ is the canonical dual basis to the basis of rational vector fields $\frac{\partial}{\partial z_1}$,..., $\frac{\partial}{\partial z_n}$.

If $\omega(\frac{\partial}{\partial z_i})$ is polynomial, $i=1$,..., $n$, we say that $\omega$ is a {\it polynomial differential $1$-form} on $K^n$, and in this case it is a $K[z]$-linear map $\mathfrak{X}(K^n)\rightarrow K[z]$, $X\mapsto \omega(X)$.  The $K$-vector space of polynomial differential $1$-forms will be denoted by $\Omega^1(K^n)$

If $f$ is a rational function in $K(z)$, we write the {\it exact} $1$-form
$$
df=\sum_{i=1}^n \frac{\partial f}{\partial z_i}d z_i
$$ 


\begin{definition}
If $\omega$ is a rational $1$-form and $f$ is a rational function such that $df\neq 0$ and $\omega\wedge df = 0$ we say that $f$ is a {\it first integral} or a {\it non-trivial constant of derivation} of $\omega$. In case of $K=\mathbb R$ or $\mathbb C$ then $\{ f(z)=constant\}$ defines a foliation of $K^n$.
If $\omega$ is a polynomial $1$-form  we say that an irreducible polynomial $F$ is {\it invariant} by $\omega$ if $F$ divides $\omega\wedge dF$ (we also say that $\{F=0\}\subset (K^a)^n$ is an {\it invariant hypersurface} for $\omega$ over the algebraic closure $K^a$ of $K$). In case of $K=\mathbb R$ or $\mathbb C$ then $\{ F=0\}$ is a {\it leaf} of $\omega$.
\end{definition}

Analogously we can define differential $r$-forms for $r\geq 1$ and the $K$-vector space of polynomial differential $r$-forms will be denoted by $\Omega^r(K^n)$.

\begin{definition}
Let $F$ be an invariant polynomial for the $1$-form $\omega$. We say that the polynomial $2$-form
$$
\Theta_F=\omega\wedge \frac{dF}{F}
$$
is a {\it cofactor} of $F$.
\end{definition}

\begin{definition}
Let $K$ be a field of arbitrary characteristic $p$ ($p=0$ or $p$ is a prime integer). The {\it field of differential constants} is the sub-field of $K(z)$ given by
$$
K(z^p)=\{f: f\in K(z), df=0 \}
$$
Note that in characteristic $0$ we have $K(z^p)=K(z^0)=K$. Otherwise, in prime characteristic $p>0$ we obtain $K(z^p)$ as the $K$-vector subspace of $K(z)$ generated by $\{g^p: g\in K(z) \}$.

We call the elements of $K(z^p)$ of {\it $\partial$-constants}.

\end{definition}

\begin{remark}
Note that $K(z^p)$ is the kernel of the $d$ operator, and $K(z)$ is finite dimensional as a $K(z^p)$-vector space.
\end{remark}

\begin{definition}
Let $F_1$,..., $F_m$ be a collection of irreducible polynomials and $\lambda_1$,..., $\lambda_m\in K(z^p)$ be $\partial$-constants. The linear combination
$$
\eta=\lambda_1\frac{dF_1}{F_1}+...+\lambda_m\frac{dF_m}{F_m}
$$
is called a {\it logarithmic} $1$-form.
\end{definition}

Note that every logarithmic $1$-form $\eta$ is {\it closed}, that is, $d\eta=0$.

\section{Logarithmic 1-forms and Residues}

Remember the Hilbert's Nullstellensatz:

\begin{thm}[\cite{Lang}, Theorem 1.5, p. 380]
Let $I$ be an ideal of $K[z]$ and let $V(I)=\{z\in (K^a)^n: f(z)=0, \forall f\in I\}$ be the algebraic variety associated to $I$, where $K^a$ is the algebraic closure of $K$. Let $P$ be a polynomial in $K[z]$ such that $P(c)=0$ for every zero $(c)=(c_1,...,c_n)\in V(I)$. Then there is an integer $m>0$ such that $P^m\in I$.
\end{thm}

We will need something about residues.

\begin{lemma}\label{L:Resíduo}
Let $g\in K[x]$, $g(0)\neq 0$, be a polynomial function in one variable, where $K$ is a field of positive characteristic $p>0$. If $\alpha \in K(x^p)$, then
$$
Res(\alpha\cdot \frac{dg}{g},0)=0
$$
\end{lemma}
\begin{proof}

We can write $g(x)=a_0+a_1x+...+a_rx^r$, where $a_0\neq 0$, and
$$
\alpha(x)=x^{sp}\frac{A(x)}{B(x)}
$$
where $s$ is an integer and $A(0)\neq 0$, $B(0)\neq 0$.

Hence
$$
\alpha\cdot \frac{dg}{g}=x^{sp}\frac{A(x)g'(x)}{B(x)g(x)}dx
$$
Since $B(0)g(0)\neq 0$ note that $\alpha\cdot \frac{dg}{g}$ has a pole at $0$ if and only if $A(x)g'(x)$ has a term $cx^{-sp-1}$ with $c\neq 0$. Also note that
$$
d(x^{-sp})=(-sp)x^{-sp-1}dx=0
$$
Hence $A(x)g'(x)$ has no term $cx^{-sp-1}$ with $c\neq 0$. We conclude that the residue is $Res(\alpha\cdot \frac{dg}{g},0)=0$.

\end{proof}

The following lemma is the positive characteristic version of Jouanolou's Lemma (\cite{Jouanolou}, Lemme 3.3.1, p. 102).

\begin{lemma}\label{L:Logarítmicas}
If $\mathcal S$ is a finite representative system of primes in $K[z]=K[z_1,...,z_n]$ (that is, $\mathcal S$ is a finite collection of distinct irreducible polynomials), then the $K(z^p)$-linear map
$$
K(z^p)^{(\mathcal S)}\longrightarrow \Omega^1_{K(z)/K}
$$
$$
(\lambda_j)_{f_j\in \mathcal S}\longmapsto \sum_{f_j\in \mathcal S} \lambda_j \frac{df_j}{f_j}
$$
is injective.
\end{lemma}
\begin{proof}
Let $\lambda_1$,..., $\lambda_r\in K(z^p)$ and $f_1$,..., $f_r\in \Omega^1_{K(z)/K}$ be such that
$$
\sum_{j=1}^r \lambda_j \frac{df_j}{f_j}=0
$$

For each $i\in \{1,2,..., r\}$, define $V_i=\{f_i=0\}\subset (K^a)^n$. By Hilbert's Nullstellensatz, there is a point 
$$
p_j\in (V_j-\cup_{i\neq j} V_i)
$$

(Suppose that for every $p\in V_j$ there is a $V_i$, $i\neq j$, such that $p\in V_i$; that is, $V_j\subset \cup_{i\neq j} V_i$. Consider $I$ the ideal generated by $\Pi_{i\neq j} f_i$. We have $V(I)=\cup_{i\neq j} V_i$. Then, by the Hilbert theorem, there exists an integer $n$ such that $f_j^n$ is in the ideal $I$, and then it $f_j$ has a common factor with  $\Pi_{i\neq j} f_i$, and it is  a contradiction since every $f_i$ is irreducible.) 

We can suppose that $p_j$ is a smooth point of $V_j$. Let $L_j$ be an affine line through $p_j$ transversal to $V_j$. Consider $u:K^a\rightarrow L_j$ be a parametrization of $L_j$ with $u(0)=p_j$ and $g_i=f_i\circ u$. For all $i\neq j$, the function $\frac{dg_i}{g_i}$ is regular at $0$, but $\frac{dg_j}{g_j}$ has a pole at $0$ with $Res(\frac{dg_j}{g_j},0)=1$. Then, from the above relation, we have
$$
\alpha_1 \frac{dg_1}{g_1}+...+\alpha_j \frac{dg_j}{g_j}+...+\alpha_r \frac{dg_r}{g_r}=0
$$
where $\alpha_i=\lambda_i\circ u$ for $i=1$,..., $r$.

If $\alpha_j\neq 0$, we can put
$$
\frac{dg_j}{g_j}=-\sum_{i\neq j} \frac{\alpha_i}{\alpha_j } \frac{dg_i}{g_i}
$$
and taking residues at $0$, using Lemma \ref{L:Resíduo} we obtain 
$$
1=-\sum_{i\neq j} Res(\frac{\alpha_i}{\alpha_j} \frac{dg_i}{g_i})=0
$$
Hence $\alpha_j=0$, that is, $\lambda_j\mid_{L_j}=0$. Since $p_j$ and $L_j$ are generic, we conclude that $\lambda_j=0$.

\end{proof}

\section{Darboux Integration Method}

We explore the method of integration developed by J.-G. Darboux in \cite{Darboux}. The original context was differential equation on $\mathbb C^2$. J.-P. Jouanolou in \cite{Jouanolou} extends the argument for $1$-forms on $K^n$, where $K$ is an algebraically closed field of characteristic $0$. Here we extend this method for an arbitrary field $K$.

\begin{prop}
Let $\omega\in \Omega^1(K^n)$ be a $1$-form, where $K$ is an arbitrary field. If there are invariant irreducible polynomials $F_1$,..., $F_m$ in $R$ and constants $\alpha_1$,..., $\alpha_m$ in $K(z^p)$ such that
$$
\alpha_1 \Theta_{F_1}+\alpha_2 \Theta_{F_2}+...+\alpha_m \Theta_{F_m}=0
$$
then there is a logarithmic $1$-form $\eta\neq 0$ such that $\omega\wedge \eta=0$. (In that case we say that $\eta$ is {\it tangent} to $\omega$.)
\end{prop}
\begin{proof}
Consider the rational $1$-form
$$
\eta=\alpha_1 \frac{dF_1}{F_1}+\alpha_2 \frac{dF_2}{F_2}+...+\alpha_m \frac{dF_m}{F_m}\neq 0
$$
(see Lemma \ref{L:Logarítmicas}).

Obviously $\eta$ is {\it closed} ($d\eta=0$). Also $\eta$ satisfies
$$
\omega\wedge \eta=\sum_{i=1}^m\alpha_i\omega \wedge \frac{dF_i}{F_i}=\sum_{i=1}^m\alpha_i \Theta_{F_i}=0
$$

\end{proof}

\begin{remark}
In case of $K=\mathbb R$ or $\mathbb C$, from the logarithmic $1$-form $\eta$ above, by integration we obtain a (multivaluated, analytic or holomorphic) first integral for $\omega$.
\end{remark}

\begin{definition}\label{D:Nk}
Let $K$ be a field of positive characteristic $p>0$, and define $S_K(d)$ to be the $K$-vector space of polynomials of degree less than or equal to $min\{p-1,d\}$. In case that $K$ is a field of characteristic $0$, define $S_K(d)=S_d$ as the $K$-vector space of polynomials of degree less than or equal to $d$. Define
$$
\widehat{\Omega}_d^r(n):= \Omega^r(K^n)\otimes S_K(d)\otimes K(z^p)
$$
and
$$
N_K(n, d, r):=dim_{K(z^p)}(\widehat{\Omega}_d^r(n))\leq dim_K(\Omega_d^r(K^n))
$$

\end{definition}

\begin{cor}\label{C:Cofatores}
Let $\omega\in \Omega^1(K^n)$ be a $1$-form of degree $d$. If $\omega$ has $N_K(n, d-1, 2)+1$ invariant irreducible polynomials, then there is a logarithmic $1$-form $\eta\neq 0$ such that $\omega\wedge \eta=0$.
\end{cor}
\begin{proof}
Suppose that there are $m=N_K(n, d-1, 2)+1$ invariant irreducible polynomials $F_1$, $F_2$,..., $F_m$.

The cofactors associated to the invariant polynomials are differential $2$-forms of degree less than or equal to $d-1$. Since the $K(z^p)$-vector space of differential $2$-forms of degree less than or equal to $d-1$ has $K(z^p)$-dimension equal to $N_K(n, d-1, 2)$, we have that the cofactors associated to the polynomials $F_i$, $i=1$,..., $m$, are $K(z^p)$-linearly dependent. Therefore, there are $\partial$-constants $\alpha_1$,..., $\alpha_m$ in $K(z^p)$ such that
$$
\alpha_1 \Theta_{F_1}+\alpha_2 \Theta_{F_2}+...+\alpha_m \Theta_{F_m}=0
$$
By the above proposition, the corollary follows.

\end{proof}

\begin{example}
Suppose $\omega=\alpha ydx+ \beta x dy$, where $\alpha$ and $\beta$ are constants in $K^*$. We can easily see that $F=x$ and $G=y$ are invariant. The cofactor associated to $x$ is given by
$$
\Theta_x=\omega\wedge \frac{dx}{x}=-\beta dx\wedge dy
$$
while the cofactor associated to $y$ is
$$
\Theta_y=\omega\wedge \frac{dy}{y}=\alpha dx\wedge dy
$$
Hence
$$
\alpha\Theta_x+\beta\Theta_y=0
$$
Consequently $\eta=\alpha \frac{dx}{x}+\beta \frac{dy}{y}$ is a logarithmic $1$-form tangent to $\omega$ (in characteristic $p=2$ the $1$-form $\eta$ is linear).

\end{example}

\begin{prop}
Let $\omega\in \Omega^1(K^n)$ be a $1$-form. If there are invariant irreducible polynomials $F_1$,..., $F_m$ in $K[z]$ and constants $\delta_1$,..., $\delta_m$ in $\mathbb Z$ (or in the prime subfield $\mathbb Z_p$ of $K$ if $char(K)=p>0$) such that
$$
\delta_1 \Theta_{F_1}+\delta_2 \Theta_{F_2}+...+\delta_m \Theta_{F_m}=0
$$
then $\omega$ has a rational (polynomial) first integral.
\end{prop}
\begin{proof}
Consider the rational $1$-form
$$
\eta=\delta_1 \frac{dF_1}{F_1}+\delta_2 \frac{dF_2}{F_2}+...+\delta_m \frac{dF_m}{F_m}\neq 0
$$
Obviously $\eta$ is {\it closed} ($d\eta=0$) and also $\eta$ satisfies
$$
\omega\wedge \eta=\sum_{i=1}^m\delta_i\omega \wedge \frac{dF_i}{F_i}=\sum_{i=1}^m\delta_i \Theta_{F_i}=0
$$

Define the rational function (or polynomial)
$$
G=F_1^{\delta_1}F_2^{\delta_2}...F_m^{\delta_m}
$$
Note that $dF_i^{\delta_i}=\delta_i F_i^{\delta_i}\frac{dF_i}{F_i}$ and then $dG=G\eta $, hence 
$$
\omega\wedge dG=G\omega\wedge \eta=0
$$

\end{proof}

\section{Darboux-Jouanolou Criterion: Proof of Theorem \ref{Main}}


Let $F_1$,..., $F_m$ be $N_K(n, d-1, 2)+2$ invariant irreducible polynomials (for the $1$-form $\omega$). By the using of the cofactors associated to the invariant polynomials $F_1$,..., $F_{m-1}$, we can construct, by Corollary \ref{C:Cofatores}, the logarithmic $1$-form
$$
\eta_1=\alpha_1 \frac{dF_1}{F_1}+\alpha_2 \frac{dF_2}{F_2}+...+\alpha_{m-1} \frac{dF_{m-1}}{F_{m-1}}\neq 0
$$
such that $\omega\wedge \eta_1=0$. Analogously, we also can obtain
$$
\eta_2=\beta_2 \frac{dF_2}{F_2}+\beta_3 \frac{dF_3}{F_3}+...+\beta_m \frac{dF_m}{F_m}\neq 0
$$
such that $\omega\wedge \eta_2=0$ with $\beta_m\neq 0$.

As we suppose $\beta_m\neq 0$, the "polar divisors" of $\eta_1$ and $\eta_2$ are distinct. Hence, by Lemma \ref{L:Logarítmicas}, there is a rational function $f\in K(z)-K(z^p)$ such that $\eta_1=f\eta_2$. Taking the exterior derivative, we obtain
$$
d\eta_1=fd\eta_2+df\wedge \eta_2
$$
Since $\eta_1$ and $\eta_2$ are closed, we conclude that $df\wedge \eta_2=0$. In this way, since $\omega\wedge \eta_2=0$ we have that $\omega\wedge df=0$. Therefore $f$ is a rational first integral for $\omega$.

\qed


\bibliographystyle{amsplain}

\end{document}